\documentclass[12pt,a4paper,dvipsnames]{article}

\usepackage{hyperref}
\usepackage{pgf}
\usepackage[utf8]{inputenc}
\usepackage{amsfonts}
\usepackage{amssymb}
\usepackage{amsthm}
\usepackage{fixmath}

\usepackage{enumitem}

\usepackage{pgfplots}
\usepackage{dsfont}
\usepackage[linesnumbered]{algorithm2e}
\usepackage{setspace}
\usepackage{extarrows}
\usepackage[makeroom]{cancel}
\usepackage{graphicx}
\usepackage{multicol}
\PassOptionsToPackage{dvipsnames}{xcolor}
\usepackage{tikz,xcolor}

\usetikzlibrary{shapes,arrows}
\usetikzlibrary{hobby}
\usetikzlibrary{decorations.markings}

\usepackage[makeroom]{cancel}

\usepackage[english]{babel}		
\usepackage[T1]{fontenc}			
\usepackage{color}         			
\usepackage{listings}					

\usepackage{geometry}
\newtheorem{theorem}{Theorem}[section]
\newtheorem{lemma}[theorem]{Lemma}

\newtheorem{definition}[theorem]{Definition}
\newtheorem{corollary}[theorem]{Corollary}
\newtheorem{remark}[theorem]{Remark}
\newtheorem{conjecture}[theorem]{Conjecture}

\title{Bounded cutoff window for the non-backtracking 
random walk on Ramanujan Graphs}
\author{Evita Nestoridi$^{\ast}$ \and Peter Sarnak$^{\ast}$}
\date{\today}

\begin{document}
\tikzset{
    arrow/.style={postaction={decorate},
        decoration={markings,mark=at position 1 with
        {\arrow[line width=0.6mm]{>}}}} 
       }
\maketitle
\begin{abstract} 
We prove that the non-backtracking random walk on Ramanujan graphs with large girth exhibits the fastest possible cutoff with a bounded window.
\end{abstract}

\phantom{.} \hspace{0.2cm}

\let\thefootnote\relax
\footnotetext{ $^{\ast}$ \textit{Princeton University, United States, E-Mail :exn@princeton.edu,  sarnak@math.princeton.edu} \\ 
\phantom{.} \hspace{0.3cm} 
}

\section{Introduction}
Fix $d \geq 3$, which we write as $p+1$, and consider $d$-regular graphs $X$ on $n$ vertices with $n \rightarrow \infty$. During the last decade, there has been a lot of interest in studying the simple random walk (SRW) and the non-backtracking random walk (NBRW) on such graphs. The focus has been to understand mixing times and related cutoff phenomena \cite{CP, Hermon, LP, LLP, ON, Sardari}. The non-backtracking random walk was introduced by Hashimoto \cite{Has}, it mixes faster, has sharper transitions and has been very useful in multiple cases \cite{ABLS, BLM, BS, FH, GK, Hermon2, HSS, OW}. We focus exclusively on the NBRW on $X$, which is defined as follows:
\[K_t(x,y)=\# \bigg \lbrace (x=x_1, x_2,\ldots, x_t=y) \bigg \vert \substack{ x_i \in X \\ x_i\sim x_{i+1}\\  x_{i-1}\neq x_{i+1}} \bigg \rbrace ,\]
where $x_i \sim x_{i+1}$ indicates that $(x_i, x_{i+1})$ is an edge.

Let $N(t):= \sum_y K_t(x,y) =(p+1)p^{t-1}$ and let $P^t(x,y)= \frac{1}{N(t)} K_t(x,y)$ be the transition matrix of the non-backtracking random walk on $G$. The total variation of $P^t_x= P^t(x, \cdot)$ from the  uniform measure is defined as 
$d_x(t):= \frac{1}{2} \sum_{y \in X} \bigg \vert P^t(x,y) - \frac{1}{n} \bigg \vert $. 
We will also consider the total variation distance when starting at the worst possible starting point
\[ d(t):=\max_{x \in X} \{ d_x(t) \} .\]
For  $0<\eta <1$, the total variation mixing time is defined as
\[t_{mix}(\eta)= \min \{t \geq 0 :  d(t) \leq \eta\}.\]
The main focus of this paper is studying the cutoff phenomenon. We say that the NBRW on $X$ exhibits cutoff at $t_n$ with window $w_n$ if
\begin{equation}\label{cutoff}
\lim_{c \rightarrow \infty} \lim_{n \rightarrow \infty} d \left(t_n - c w_n \right)= 1 \mbox{ and } \lim_{c \rightarrow \infty} \lim_{n \rightarrow \infty} d \left(t_n + cw_n \right)= 0.
\end{equation}

If $N(t) \leq n$ one checks that
$d_x(t)= \frac{U_x(t)}{n},$
where $U_x(t)$ is the number of vertices that are not reached by the walk at time t, when starting at $x$. Hence, 
\begin{equation*}
d(t) \geq 1- \frac{N(t)}{n},
\end{equation*}
if $N(t) \leq n$, which implies that 
\begin{equation}\label{lowercut}
t_{mix}(1- \eta) \geq \log_p n - \log_p \eta^{-1}. 
\end{equation}
This gives an absolute lower limit in \eqref{cutoff} for the cutoff time $t_n=\log_p n$ and bounded $w_n$ and we are interested in graphs $X$ for which this $t_n$ is indeed the cutoff time for the NBRW.

We will search for such $X$ among different types of expanders. For $\lambda<d$ an $(n, d , \lambda)$ graph $X$ is a $d$ regular graph on $n$ nodes for which the eigenvalues $\{\lambda_{j} \}_{j=0}^n$ of the adjacency matrix of $X$ satisfy
\[   
     \begin{cases}
       \lambda_0= d:=p+1 &\quad\text{if} \quad  j = 0 \\
      \vert \lambda_j \vert \leq \lambda &\quad\text{if} \quad  j \neq 0 .
     \end{cases}
\]
If $\lambda= 2 \sqrt{p}$ then $X$ is called a Ramanujan graph.

The key results in this direction are due to Lubetzky and Sly \cite{LP} and Lubetzky and Peres \cite{LS}. In the first, it is shown that for the random $d$-regular graph 
\begin{equation}\label{upcut}
t_{mix}(\eta) \leq \log_p n + 3 \log \eta^{-1} +4. 
\end{equation}
Equations \eqref{lowercut} and \eqref{upcut} show that for such graphs the NBRW exhibits total variation cutoff at $\log_p n$ with a bounded window. In the second one, $X$ is assumed to be Ramanujan, and they show that the NBRW on any Ramanujan graph exhibits cutoff at $\log_p n$, but whether it occurs with a bounded window is not resolved. More precisely, they show that  
\begin{equation*}
d(t) \leq \frac{1}{\log n}, 
\end{equation*}
for every $t > \log_p n + 3\log_p \log n$. As a corollary they also prove a purely combinatorial fact about the almost diameter. For Ramanujan graphs, for any $x$ we have that
\begin{equation}\label{diameter}
 \# \{ y \in X : \vert \mbox{dist}(x,y) - \log_p n \vert > 3 \log_p \log n\}= o(n). 
\end{equation}
The same result was independently proven by Sardari in \cite{Sardari}. 

Our first result shows that the NBRW on a Ramanujan graph with large girth $g$ exhibits cutoff with a bounded window.
\begin{theorem}\label{Ram2}
Fix $\delta>0$. The NBRW on a Ramanujan graph with $g \geq \delta \log_p n$ satisfies
\[ t_{mix}(\varepsilon) \leq \log_p n +2 \log_p \varepsilon^{-1} + 2 \log_p (2 +20 \delta^{-1}),  \]
for every $\varepsilon>0$.
\end{theorem}

\begin{remark}
The girth condition of Theorem \ref{Ram2} is satisfied for Ramanujan graphs \cite{LPS} with $\delta= 2/3 $. This shows that the NBRW on these graphs exhibits cutoff with a bounded window, which was one of our goals in this note.
\end{remark}
It is important to note that most of the examples (other than the result of Lubetzky and Sly \cite{LS}) that are known where the cutoff window is bounded are non-local Markov chains, such as riffle shuffles \cite{BD} and random transvections \cite{Hild}.

Next, we discuss what can be said about cutoff if we drop the Ramanujan condition. Writing the eigenvalues in form
\[\lambda_j = 2 \sqrt{p} \cos \theta_j ,\]
where for $\vert \lambda_j  \vert \leq 2 \sqrt{p}$, we have a unique $\theta_j \in [0, \pi]$, and otherwise for the "exceptional" eigenvalues we choose $\theta_j$ uniquely in the from 
\[   
     \begin{cases}
       \theta_j= i \phi_j \log p&\quad\text{if} \quad  \lambda_j  > 2 \sqrt{p} \\
       \theta_j = \pi + i \psi_j \log p  &\quad\text{if} \quad  \lambda_j  < - 2 \sqrt{p},
     \end{cases}
\]
with $\phi_j, \psi_j \in (0, \frac{1}{2})$.
\begin{definition}\label{S-X}
A sequence of graphs $X$ is said to satisfy the density hypothesis if for every $0\leq \alpha < 1/2$ and $\varepsilon >0$, the number of exceptional eigenvalues $M$ satisfies
\[M(\alpha,X):= \# \{ j: \phi_j \geq \alpha \} + \# \{ j: \psi_j \geq \alpha \}  \ll_{\varepsilon} n^{1-2 \alpha + \varepsilon}. \]
\end{definition}
For a discussion of this density hypothesis see \cite{Sa} and \cite{GoKa}. The point is that this density can often be established in cases where the Ramanujan is not known or even fails. 

In \cite{BoLa} and \cite{GoKa} it is shown that the density together with the assumption that $X$ is an expander suffice to show that the SRW on $X$ exhibits cutoff at $\frac{p+1}{p-1} \log_p n$. We show that the shortest possible cutoff applies to the NBRW.

\begin{theorem}\label{SX}
Let $X$ be a homogeneous sequence  (that is the automorphisms act transitively on the vertices)  of $(n,d, \lambda)$ expander graphs which satisfy the density hypothesis. Then the NBRW on $X$ exhibits cutoff at $\log_p n$. That is, 
\[ d((1+ \eta) \log_p n) \rightarrow 0,\]
for every $\eta>0$.
\end{theorem}

The next results focus on the diameter of $(n, d, \lambda)$ graphs and strengthens \eqref{diameter}.  
Let $\mathcal{N}_x(\ell)$ be the number of vertices $y\in X$ such that $d(x,y)> \ell$.
\begin{theorem}\label{expa}
Let $X$ be an $(n, d , \lambda)$ graph; then for $\xi>0$ we have that  
\[\max_{x\in X} \bigg \{  \frac{1}{n}\mathcal{N}_x \left( \frac{1}{2}\log_b n + \xi \right) \bigg\}\leq \frac{4}{b^{2 \xi}} ,\]
where $b= \frac{d}{ \lambda} + \sqrt{ \left(\frac{d}{ \lambda} \right)^2 -1 } $.
\end{theorem}
We note that if we choose $\xi$ (bounded) so that $4 b^{-2 \xi}<1/2$, then given $x,y \in X$ we can find a common $z$ with $d(x,z) <\frac{1}{2}\log_b n + \xi  $ and $d(y,z) <\frac{1}{2}\log_b n + \xi  $. Therefore, $d(x,y) <\log_b n +2 \xi  $. This shows that the diameter is at most $\log_b n +2 \xi  $. This matches the bounds for the diameter that were derived in \cite{LPS} for Ramanujan graphs and in \cite{CFM} for $(n, d, \lambda)$ graphs. As in these papers, a crucial element in the analysis are the Chebychev polynomials of the first kind.
 
Let $p=d-1$. For the case where $X$ is Ramanujan, we have that $\lambda=2 \sqrt{p}$ and $b=\sqrt{p}$. Theorem \ref{expa} gives the following.
\begin{corollary}\label{Ram}
Let $X$ be a Ramanujan graph on $n$ vertices, then for $\xi>0$ we have that
\[ \max_{x\in X} \bigg \{  \frac{1}{n}\mathcal{N}_x(\log_p n + \xi ) \bigg \}\leq \frac{4}{p^{ \xi}} .\]
\end{corollary}
\begin{remark}
Corollary \ref{Ram} gives a bounded window strengthening \eqref{diameter} and if it is not optimal, it is very close to being so. In particular, it allows one to replace the $3 \log_p \log n$ term in \eqref{diameter} by any function $f(n)$ which goes to infinity with $n$. 
\end{remark}

In the context of $d$-regular graphs , the almost diameter bound of Corollary \ref{Ram} is essentially the smallest it could be among all such graphs. On the other hand, the bound $2 \log_p n +4 $ for the diameter of a Ramanujan graph is probably not optimal. The random $d$-regular graph has diameter $(1+o(1)) \log_p n$ (see \cite{BF}), however the \cite{LPS}  Ramanujan graphs can have diameter at least $\frac{4}{3} \log_p n$, as was shown in \cite{Sardari}. We expect that this $\frac{4}{3} \log_p n$ is an upper bound for the diameter of a Ramanujan graph.

As is standard in cutoff analysis, our proofs involve the $\ell^2$ distance
$\Vert  P^t_x -U\Vert_2^2= \sum_{y \in X} \bigg \vert P^t(x,y) - \frac{1}{n} \bigg \vert^2  $
and its average over $x$
\[d_{2}(t):= \frac{1}{n} \sum_x \Vert  P^t_x -U\Vert_2^2.\]
Note that if $X$ is homogeneous, then $d_2(t)=\Vert  P^t_x -U\Vert_2^2$ for all $x$, as are all of the quantities defined in terms of the starting point $x$.

For the case of reversible Markov chains, such as the SRW on $X$, one can express $d_2(t)$ in terms of the eigenvalues and eigenfunctions of the transition matrix (see chapter $12$ of \cite{LPb}). Studying the spectrum of the transition has been a powerful tool for proving cutoff for many well known Markov chains, such as \cite{DS, Hild, BN}. We make judicious use of Chebychev polynomials and the eigenvalues and eigenfunctions of the adjacency matrix of $X$ to prove our results, and avoid using the NBRW on the edges of the graph.

Our analysis leads to the following basic conjecture.
\begin{conjecture}\label{conj}
If $X$ is a sequence of Ramanujan graphs and $t < 2 \log_p n$, then
\begin{equation}\label{13}
 d_{2}(t) \sim \frac{1}{N(t)}
\end{equation}
as $n \rightarrow \infty$. 
\end{conjecture}
This is consistent with the model that in this window the $N(t)$ end points of walks of length $t$
 are placing themselves at random among the $n$ vertices.

Our proofs involve approximations to \eqref{13}. The source of the gain being that the Kesten measure on $[-2 \sqrt{p}, 2 \sqrt{p}]$ vanishes to second order at $-2 \sqrt{p}$ and $2 \sqrt{p}$ (see \eqref{new}). In \cite{AGV} it is proven that the probability measure supported on $[-2 \sqrt{p}, 2 \sqrt{p}]$ corresponding to the eigenvalues of a Ramanujan graph, converges to the Kesten measure as $n \rightarrow \infty$. Conjecture \ref{conj} requires that this convergence holds with polynomials of degree as large as $\log_p n$. In a forthcoming paper \cite{SZ} this convergence and in particular Conjecture \ref{conj} is established for various arithmetic Ramanujan graphs. Our Conjecture \ref{conj} implies that the NBRW on these Ramanujan graphs exhibit cutoff with an explicit and tight bounded window, namely
\[ t_x(\varepsilon) \leq \log_p n + 2 \log_p \varepsilon^{-1},\]
for every starting point $x$.

\section{Preliminaries}
Let $X$ be a connected, $d$ regular graph on $n$ vertices, where $d$ is fixed. Let $A$ denote the adjacency matrix of $X$. $A$ is a symmetric matrix with eigenvalues
\[-d \leq \lambda_{n-1} \leq \ldots \leq \lambda_1<\lambda_0=d .\] 
Denote the corresponding orthonormal basis of eigenfunctions as $f_{n-1}, \ldots ,f_0 ,$ with $f_0(x)=\frac{1}{\sqrt{n}}$ for every $x\in X$. The fact that the $\{f_j\}$ are orthonormal means that 
\begin{equation} \label{orthog}
 \sum_{x \in X} f_i(x) f_{j}(x)=\delta_{i,j}.
\end{equation}
The fact that $\{f_j\}$ is an orthonormal basis gives that 
\begin{equation*}
\delta_x(y)= \sum_{j=0}^{n-1} \langle f_j, \delta_x \rangle f_j(y),
\end{equation*}
which translates to
\begin{equation}\label{sum}
\delta_x(y)= \sum_{j=0}^{n-1}f_j(x) f_j(y).
\end{equation}

When considering the $t$-th power of $A$, we have that the $(x,y)$ entry $A^t(x,y)$ is equal to the number of walks of length $t$ starting at $x$ and ending at $y$.
Let $P$ be a polynomial of the form 
\[P(x)= a_0 + a_1 x+ \ldots + a_{\ell}x^{\ell}.\]
We have that the matrix $P(A)$ can be expressed as
\begin{equation*}
P(A)(x,y)= \sum_{j=0}^{n-1} P(\lambda_j)f_j(x) f_j(y).
\end{equation*}
The key quantity that we estimate is the variance $W$ with respect to $P,$ defined as 
\begin{equation}\label{ee}
W(P,x):= \sum_{y } \left( P(A)(x,y) - \frac{P(\lambda_0)}{n} \right)^2,
\end{equation}
which by \eqref{orthog} is equal to the spectral sum 
\begin{equation}\label{var}
 \sum_{j \neq 0} \vert P(\lambda_j) \vert^2 f_j^2(x).
\end{equation}
\section{The almost diameter}
To estimate the almost diameter of $X$, we use the following key lemma.
\begin{lemma}\label{poly}
Let $\ell(P)$ be the degree of $P$, then for any $x$
\begin{equation*}
\left( \frac{P(\lambda_0)}{n}\right)^2 \mathcal{N}_x(\ell(P)) \leq  \max_{\lambda\neq \lambda_0 } \{  \vert P(\lambda ) \vert^2 \},
\end{equation*}
where $\mathcal{N}_x(\ell(P))$ is the number of vertices $y\in X$ such that $d(x,y)> \ell(P)$.
\end{lemma}
\begin{proof}
First of all, we note that since $A^t(x,y)$ is equal to the number of walks of length $t$ starting at $x$ and ending at $y$, we have that 
\begin{equation}
\mbox{for every $x, y \in X$, if }d(x,y) > \ell(P) \mbox{ then } P(A)(x,y)=0. 
\end{equation}
Combining this with \eqref{ee} and \eqref{var} we have that
\begin{equation}\label{ineq}
\sum_{y : d(x,y)> \ell(P)} \left( \frac{P(\lambda_0)}{n}\right)^2 \leq W(P,x) \leq \max_{\lambda \neq \lambda_0 } \{  \vert P(\lambda ) \vert^2 \} \sum_{j \neq 0} \vert f_j(x) \vert^2 .
\end{equation}
Equation \eqref{sum} gives that $\sum_{j \neq 0} \vert f_j(x) \vert^2 \leq 1 $, which finishes the proof.
\end{proof}

\subsection{Chebychev polynomials of the first kind}
Let $T_{\ell}$ be the Chebychev polynomials of the first kind of degree $\ell$, that is $T_{\ell}(x)= \cos (\ell \arccos x)$ and therefore $T_{\ell}(x ) \in [-1,1]$ for every $x\in [-1,1]$.

\begin{lemma}\label{Che}
For $\lambda \leq \lambda_0$, the Chebychev polynomials of the first kind satisfy 
$$T_{\ell} \left(\frac{\lambda_0}{\lambda} \right) \geq \frac{b^{\ell}}{2},$$
where $b=\left( \frac{\lambda_0}{ \lambda} + \sqrt{ \left(\frac{\lambda_0}{ \lambda} \right)^2 -1 } \right)$.
\end{lemma}
\begin{proof}
Using the fact that $\cos \theta= \frac{e^{i \theta} +e^{-i \theta}}{2}$, we can write $\lambda_0=\lambda \cos \theta_0$, where $\theta_0= i \log \left( \frac{\lambda_0}{ \lambda} + \sqrt{ \left(\frac{\lambda_0}{ \lambda} \right)^2 -1 } \right)$. This gives that 
$$T_{\ell} \left(\frac{\lambda_0}{\lambda} \right) =T_{\ell} \left( \cos \theta_0 \right)= \frac{1}{2} (b^{\ell} + b^{-\ell}) \geq  \frac{b^{\ell}}{2}.$$
\end{proof}

\subsection{The almost diameter for expanders}
In this section, we present the proof of Theorem \ref{expa} concerning the almost diameter of $(n, d, \lambda)$ graphs. 
\begin{proof}[Proof of Theorem \ref{expa}]
Let $T_{\ell}$ be the Chebychev polynomial of the first kind of degree $\ell$. We apply Lemma \ref{poly} to the polynomial
\[P(x)= T_{\ell}\left(\frac{x}{\lambda} \right),\]
where $\ell$ will be determined later.
The right hand side of the equation in Lemma \ref{poly} satisfies that
\begin{equation}\label{one}
\max_{\lambda_i \neq \lambda_0 } \{  \vert P(\lambda_i ) \vert^2  \}\leq 1,
\end{equation}
since all $\lambda_i \neq \lambda_0$ satisfy that $\vert \lambda_i \vert \leq \lambda$ and $T_{\ell}(x)= \cos (\ell \arccos x)$
for $x \in [-1,1]$. At the same time, Lemma \ref{Che} gives that
\begin{equation}\label{two}
\left( P(\lambda_0)\right)^2= \left( T_{\ell}\left(\frac{\lambda_0}{\lambda} \right)\right)^2\geq \frac{b^{2\ell}}{4}.
\end{equation}
Lemma \ref{poly} and equations \eqref{one} and \eqref{two} give that 
\begin{equation}\label{final}
\frac{1}{n}\mathcal{N}_x(\ell) \leq  \frac{4n}{b^{2\ell}}.
\end{equation}
Let $\xi>0$ be as in Theorem \ref{Ram} and set $\ell=\frac{1}{2} \log_{b}n + \xi$. Then equation \ref{final} gives the desired result.
\end{proof}

\section{The mixing time for the non-backtracking random walk}
In this section, we present our results concerning the mixing time of the NBRW on $X$. 

\subsection{Chebychev polynomials of the second kind}
The NBRW can be expressed in terms of the Chebychev polynomials of the second kind. In this section, we explain this connection and we prove some useful properties for the Chebychev polynomials of the second kind.

Let $U_{\ell}$ be the Chebychev polynomials of the second kind of degree $\ell$, defined as
\[U_{\ell}(\cos \theta) =\frac{\sin \left((\ell+1) \theta \right)}{\sin \theta}. \]
The Chebychev polynomials of the second kind satisfy the following recurrence relation:
\[\begin{cases}
U_0(x)&=1\cr
U_1(x)&=2x\cr
U_{\ell+1}(x) &= 2x U_{\ell}(x) - U_{\ell-1}(x).
\end{cases}\]
Set 
\begin{equation} \label{defini}
P_{\ell}(x)= p^{\ell/2} U_{\ell}\left( \frac{x}{2 \sqrt{p}} \right). 
\end{equation}
\begin{lemma}\label{stype}
Let $A$ be the adjacency matrix of a regular graph. We have that
\[P_{\ell}(A)(x,y)= \sum_{0 \leq j \leq \ell/2} K_{\ell -2j} (x,y),\]
where $K_t(x,y)$ is the number of non-backtracking random walks of length $t$ from $x$ to $y$.
\end{lemma}
\begin{proof}
The two sides have the following generating function
\[\sum_{\ell=0}^{\infty} P_{\ell} t^{\ell} = \frac{1}{1-At +p t^2}, \]
and therefore they are equal. For more details, we refer to Lemma 1.4.3 of \cite{DSV}.
\end{proof}

We start with the following lemma. Set $ \lambda_j=2 \sqrt{p}\cos \theta_j$. Notice that $\theta_0= i \log \sqrt{p}$ and therefore 
\begin{equation}\label{l0}
P_{\ell}(\lambda_0)= \frac{p^{\ell+1} -1}{p-1}.
\end{equation}

\begin{lemma}\label{s2}
Let $g$ be the girth of $X$ and let $\ell \leq g/5$.   For $n$ large enough, we have that
\[\sum_{j=1}^{n-1} \left( U_{\ell}(\cos \theta_j) \right)^2 f_j^2(x) \leq 2,\]
for every $x \in X$.
\end{lemma}
\begin{proof}
Since  $\ell < g$, the $\ell$ first steps of the NBRW on $X$ are the same as the $\ell $ first steps on a $d$ regular tree. Therefore,
 \[  \sum_{1 \leq j \leq \ell/2} K_{\ell -2j} (x,y) = \begin{cases} 
          1 & d(x,y)\leq \ell  \mbox{ and } d(x,y) \equiv \ell \mod 2,\\
         0 & \mbox{otherwise.}
       \end{cases}
    \]
Combined with \eqref{ee}, \eqref{l0} and Lemma \ref{stype}, this gives that
\begin{align*}
W(P_{\ell},x)&=   \sum_{\substack{ d(x,y)> \ell \mbox{ or} \\ d(x,y) \equiv \ell +1 \mod 2}} \left( \frac{p^{\ell+1} -1}{n(p-1)} \right)^2 + \sum_{\substack{ d(x,y)\leq \ell  \\ d(x,y) \equiv \ell \mod 2}} \left( 1- \frac{p^{\ell+1} -1}{n(p-1)} \right)^2  \cr
& \leq  \frac{1}{n} \left( \frac{p^{\ell+1} -1}{p-1} \right)^2 +  \sum_{\substack{ d(x,y)\leq \ell  \\ d(x,y) \equiv \ell \mod 2}}  1\cr
&\leq   \frac{1}{n}  \left( \frac{p^{\ell+1}-1}{p-1}\right)^2+ \left( \frac{p^{\ell+1}-1}{p-1}\right) .
\end{align*}
For $n $ large, we use the fact that $2 \ell+2 \leq \frac{2}{5}g +2 \leq \frac{4}{5} \log_p n+2$ to get that
\begin{align}
\label{c}W(P_{\ell},x)&\leq \frac{p^{\ell+1}}{p-1} .
\end{align}
Equations \eqref{var} and \eqref{defini} give that
\[\sum_{j=1}^{n-1} \left( U_{\ell}(\cos  \theta_j) \right)^2 f_j^2(x) \leq \frac{p}{p-1} \leq 2,\]
as desired.
 \end{proof}
\subsection{The non-backtracking random walk}
The first lemma gives $K_t$ as an explicit polynomial in $A$ (see also \cite{CA} and \cite{ABLS}). 
\begin{lemma}\label{polynomial}
Set $Q_t(x)= p^{t/2} \left( \frac{p-1}{p} U_t \left(\frac{x}{2 \sqrt{p}}\right) + \frac{2}{p} T_t \left(\frac{x}{2 \sqrt{p}} \right)  \right)$. We have that
\[Q_t(A)(x,y)=K_t(x,y),\]
for every $x,y \in X$.
\end{lemma}
\begin{proof}
Using Lemma \ref{stype}, we can write that
\begin{align}
\label{k}K_t(x,y)= P_t(A)(x,y)-P_{t-2}(A)(x,y).
\end{align}
Using the following relationship between Chebychev polynomials of the two types
\[U_{t}=U_{t-2} + 2 T_t\]
and \eqref{defini}, we can rewrite \eqref{k} as
\begin{align}
\label{k2}K_t(x,y)= p^{t/2} \left( \frac{p-1}{p} U_t \left(\frac{A}{2 \sqrt{p}}\right) + \frac{2}{p} T_t \left(\frac{A}{2 \sqrt{p}} \right)  \right).
\end{align}
In other words, 
\begin{align}
Q_t(A)(x,y)=K_t(x,y).
\end{align}
\end{proof}
 We now use Lemma \ref{polynomial} to write the following expression for the variance.
\begin{align} 
W(Q_t(A),x)&=p^t \sum_{j \neq 0} \left(\frac{p-1}{p} \frac{\sin ((t+1) \theta_j)}{\sin  \theta_j} + \frac{2}{p} cos (t \theta_j) \right)^2f_j^2(x)\cr
&\label{last5} \leq  p^t  \left( t+1 \right)^2,
\end{align}
which is the bound given in Lubetzky and Peres \cite{LP}. As they note in Remark 3.7 of \cite{LP}, in order to get rid of the factor $(t+1)^2$ in \eqref{last5}, one needs some control on the distribution of the $\theta_j$.To do so, we assume a lower bound on the girth $g$ of $X$.

\begin{lemma} \label{tv}
Fix $\delta>0$ and assume that $X$ has girth $g \geq \delta \log_p n$ and is Ramanujan, then
\[W(Q_t(A),x) \leq 12 \left(  \frac{10}{\delta}+1 \right)^2p^t, \]
for $\log_p n \leq t \leq 2 \log_p n$.
\end{lemma}
\begin{proof}
Set $k= \lfloor \frac{10}{\delta}\rfloor +1$. For $t \in [\log_p n, 2 \log_p n]$, write $t+1$ as $m k + r$ with $0 \leq r <k$. Notice that then
$m \leq \frac{2}{k} \log_p n \leq g/5$ and so we can apply Lemma \ref{s2} with this $m$. 
According to Lemma \ref{polynomial} with $\lambda_j= 2 \sqrt{p} \cos \theta_j,$ we have that
\begin{align} 
W(Q_t(A),x)&=p^t \sum_{j \neq 0} \left(\frac{p-1}{p} \frac{\sin ((t+1) \theta_j)}{\sin  \theta_j} + \frac{2}{p} cos (t \theta_j) \right)^2f_j^2(x)\cr
& \label{cut} \leq 2 p^t \sum_{j \neq 0} \left( \left( \frac{\sin ((mk+r) \theta_j)}{\sin  \theta_j} \right)^2+ \frac{4}{p^2} cos^2 (t \theta_j) \right)f_j^2(x)
\end{align}
Using standard trigonometric identities we have that
\begin{align}
\eqref{cut}& \leq 2 p^t \sum_{j \neq 0}\left( \frac{\sin (mk \theta_j) \cos r \theta_j+ \cos (mk \theta_j) \sin (r \theta_j)}{\sin  \theta_j} \right)^2f_j^2(x) + 8p^{t-2}\cr
&\leq 4 p^t \sum_{j \neq 0} \left( \left( \frac{\sin (mk \theta_j) }{\sin  \theta_j} \right)^2 + \left(\frac{\sin (r \theta_j) }{\sin  \theta_j} \right)^2 \right)f_j^2(x) + 8p^{t-2}\cr
&\leq 4 p^t \sum_{j \neq 0} \left( \left( \frac{\sin (mk \theta_j) }{\sin  \theta_j} \right)^2 + r^2 \right)f_j^2(x) + 8p^{t-2}\cr
&\label{last1} \leq 4 p^t \sum_{j \neq 0} \left( \left( \frac{\sin (mk \theta_j)}{\sin  \theta_j} \right)^2  \right)f_j^2(x) + 4r^2p^t+ 8p^{t-2}
\end{align}
We use the fact that
\[\bigg \vert \frac{\sin (mk \theta_j)}{\sin \theta_j} \bigg \vert= \bigg \vert \frac{\sin (mk \theta_j)}{\sin(m \theta_j)} \frac{\sin(m \theta_j)}{\sin \theta_j} \bigg \vert \leq k    \bigg \vert  \frac{\sin(m \theta_j)}{\sin \theta_j} \bigg \vert \]
and Lemma \ref{s2} to get
\[4 p^t \sum_j \left( \frac{\sin (mk \theta_j)}{\sin  \theta_j} \right)^2 f_j^2(x) \leq 8k^2 p^t . \]
Equation \eqref{last1} gives that
\[W(Q_t(A),x) \leq 12 k^2 p^t \leq 12\left(\frac{10}{\delta} +1\right)^2 p^t,\]
since $0\leq r <k$. This completes the proof of Lemma \ref{tv}.
\end{proof}

\subsection{The bounded window}\label{bd}
In this section, we present the proof of Theorem \ref{Ram2}. Let $P $ be the transition matrix of the non-backtracking random walk on $X$. For $t \geq 0$, we have that $P^t(x,y)=\frac{1}{(p+1)p^{t-1}}K_t(x,y)$.  
Therefore, applying Cauchy-Schwartz we get that
\begin{align*}
4 d^2_x(t) & \leq  \sum_y \bigg \vert \frac{1}{(p+1)p^{t-1}}K_t(x,y)- \frac{1}{n}   \bigg \vert^2.
\end{align*}
Since $Q_t(d)= (p+1)p^t$, equation \eqref{ee} gives that
\begin{align*} 
W(Q_t(A),x)&= \sum_y \left( K_t(x,y) - \frac{(p+1) p^{t}}{n} \right)^2.
\end{align*}
Therefore,
\begin{equation}
4 d^2_x(t) \leq 
\label{ff}\frac{n}{(p+1)^2p^{2t-2}} W(Q_t(A),x).
\end{equation}
Using Lemma \ref{tv}, we get that for $\log_p n \leq t \leq 2 \log_p n,$
\begin{align}
 d_x(t) & \leq\frac{n^{1/2}}{2(p+1)p^{t-1}}( W(Q_t(A),x))^{1/2}\cr
& \leq 2 \left( 1+  \frac{10}{\delta} \right) \left(\frac{n}{p^t}\right)^{1/2}.
\end{align}
By taking $t= \log_p n + 2 \log_p \varepsilon^{-1} +2 \log_p \left( 2(1+\frac{10}{\delta}) \right),$ we get that
\[ d_x(t) \leq \varepsilon,\]
and this holds uniformly for $x \in X$.

\section{The density hypothesis}

Let $X$ be an $(n, d ,\lambda)$ graph satisfying the density property, as defined in \ref{S-X}. The goal of this section is to prove that the mixing time of the non-backtracking random walk on $X$ is at most $(1+ \eta) \log_p n$ for every $\eta>0$.

The following Lemma is key to proving Theorem \ref{SX}.
\begin{lemma}\label{eps}
Let $X$ be an $(n, d ,\lambda)$ expander sequence, that satisfies the density hypothesis. We set $I_n=\sum_{j =1}^{n-1} p^{- [\frac{1}{2}- \phi_j]2t} +\sum_{j =1}^{n-1} p^{- [\frac{1}{2}- \psi_j]2t} $. Then 
$$\lim_{n \rightarrow \infty}I_n=0,$$ 
if $t\geq(1+\eta) \log_p n $.
\end{lemma}
\begin{proof}
Since $X $ is an expander, we have that there is $\delta_1>0$ such that $0 \leq \phi_j \leq \frac{1}{2} - \delta_1$. We express the sums in $I_n$ in terms of the function $M$ in Definition \ref{S-X};
\[ \sum_{j =1}^{n-1} p^{- [\frac{1}{2}- \phi_j]2t}  =-\int^{\frac{1}{2}- \delta_1}_0 p^{- [\frac{1}{2}- \alpha]2t} dM(\alpha). \]
Integrating by parts, we get that 
\[ \sum_{j =1}^{n-1} p^{- [\frac{1}{2}- \phi_j]2t}  = M(0)p^{-t} - M\left(\frac{1}{2}- \delta_1\right) p^{-2 t \delta_1}   + 2 t \int^{\frac{1}{2}- \delta_1}_0 p^{- [\frac{1}{2}- \alpha]2t} M(\alpha)d \alpha. \]
The density hypothesis asserts that $M(\alpha) \ll_{\varepsilon} n^{1-2 \alpha + \varepsilon}$ for $0\leq \alpha \leq 1/2$. Therefore, 
\begin{align*}
\sum_{j =1}^{n-1} p^{- [\frac{1}{2}- \phi_j]2t}  &\ll_{\varepsilon}n^{1 + \varepsilon}p^{-t}   + 2 t \int^{\frac{1}{2}- \delta_1}_0 p^{- [\frac{1}{2}- \alpha]2t} n^{1-2 \alpha + \varepsilon} d \alpha\\
&= n^{1 + \varepsilon}p^{-t} \left( 1 + 2t \int^{\frac{1}{2}- \delta_1}_0\left( \frac{ p^{2t}}{ n^{2 }} \right)^{\alpha} d \alpha  \right).
\end{align*}
Since $t > \log_p n$, we have that
\begin{align*}
\sum_{j =1}^{n-1} p^{- [\frac{1}{2}- \phi_j]2t}  &\ll_{\varepsilon} n^{1 + \varepsilon}p^{-t} \left( 1 + 2t \left(\frac{p^t}{n} \right)^{1- 2 \delta_1}  \right)\\
& \ll_{\varepsilon} n^{1 + \varepsilon}p^{-t} + 2t n^{\varepsilon} \left( \frac{n}{p^t} \right)^{2 \delta_1}.
\end{align*}
Therefore, 
\begin{align}
\sum_{j =1}^{n-1} p^{- [\frac{1}{2}- \phi_j]2t}  &\ll_{\varepsilon} n^{1 + \varepsilon}p^{-t} + 2t n^{\varepsilon} \left( \frac{n}{p^t} \right)^{2 \delta_1}.
\end{align}
We can get a similar bound for $\sum_{j =1}^{n-1} p^{- [\frac{1}{2}- \psi_j]2t}$. 
Since $\varepsilon>0$ is arbitrarily small and $\delta_1>0$ is fixed, it follows that $\lim_{n \rightarrow \infty}I=0$ if $t\geq(1+\eta) \log_p n $.
\end{proof}

We are now ready to prove Theorem \ref{SX}.
\begin{proof}[Proof of Theorem \ref{SX}]

We recall that when $\lambda_j >2 \sqrt{p}$, we have that $\theta_j= i \phi_j \log p$ for $\phi_j \in (0, 1/2]$. Then,
\begin{equation}\label{c}
\vert \cos (t \theta_j) \vert = \bigg \vert \frac{1}{2} \left( p^{t \phi_j} + p^{-t \phi_j} \right) \bigg \vert \leq p^{t \phi_j} 
\end{equation}
and
\begin{equation}\label{s}
\bigg \vert U_t \left( \frac{\lambda_j}{2 \sqrt{p}} \right) \bigg \vert = \bigg \vert \frac{p^{(t+1)\phi_j} - p^{-(t+1)\phi_j}}{p^{\phi_j} - p^{-\phi_j}}  \bigg \vert \leq p^{(t+2)\phi_j} .
\end{equation}
We can get similar bounds in terms of the $\psi_j$ for the case $\lambda< - 2 \sqrt{p}$.
\begin{align*}  
W(Q_t(A),x)&=p^t \sum_{j \neq 0} \left(\frac{p-1}{p} \frac{\sin ((t+1) \theta_j)}{\sin  \theta_j} + \frac{2}{p} cos (t \theta_j) \right)^2f_j^2(x)
\end{align*}
Summing over $x$ and using the fact that $X$ is homogeneous, we have that
\begin{align} 
\label{v2}W(Q_t(A),x) & \leq   \frac{p^t}{n} \sum_{j =1} \left(\frac{p-1}{p} \frac{\sin ((t+1) \theta_j)}{\sin  \theta_j} + \frac{2}{p} \cos (t \theta_j) \right)^2.
\end{align}
We set
\[  \phi_j'= 
     \begin{cases}
      \phi_j &\quad\text{if} \quad  \lambda_j  > 2 \sqrt{p}, \\
       \psi_j \ &\quad\text{if} \quad  \lambda_j  < - 2 \sqrt{p}.
     \end{cases}
\]
Considering the terms corresponding to all  $\vert \lambda_j \vert \leq 2 \sqrt{p}$ and using equations \eqref{c} and \eqref{s}, we have that
\begin{align} 
\label{v3}W(Q_t(A),x) & \leq  p^t  \left( t+1 \right)^2+ \frac{p^t}{n} \sum_{j =1}^{n-1} \left(\frac{p-1}{p} p^{(t+2)\phi_j'}  + \frac{2}{p}p^{t \phi_j'} \right)^2\\
&\leq  p^t  \left( t+1 \right)^2+\frac{p^t}{n}  \sum_{j =1}^{n-1}  \left(\frac{p+1}{p} p^{(t+2)\phi_j'}\right)^2\\
&\leq  p^t  \left( t+1 \right)^2+3p^2\frac{p^{t}}{n}  \sum_{j =1}^{n-1} p^{2t \phi_j'}.
\end{align} 
Plugging this into the $\ell^2$ bound, we have that
\begin{align}
 d_x(t) & \leq\frac{n^{1/2}}{2(p+1)p^{t-1}}( W(Q_t(A),x))^{1/2}\cr
& \leq \frac{1}{2} \left(np^{-t}  \left( t+1 \right)^2+3 p^2\sum_{j =1}^{n-1} p^{- [\frac{1}{2}- \phi_j']2t} \right)^{1/2},
\end{align}
for every $x \in X$. Lemma \ref{eps} finishes the proof of Theorem \ref{SX}.
\end{proof}
\begin{remark}
In Theorem \ref{SX}, if $X$ is not homogeneous then in as much as we summed over all $x \in X$ in the proof, the result remains true for almost all $x$ in place of all $x$.
\end{remark}

\section{Conjecture \ref{conj}}
We end with some comments about our conjectured asymptotics  of the variance for Ramanujan graphs. The spectral expansion \eqref{var}  gives that 
\begin{align}
W_2(t) :&= \frac{1}{n} \sum_{x \in X} W(Q_t,x )\cr
&=\sum_{j \neq 0} Q_t^2(2 \sqrt{p} \cos \theta_j)\cr
&\label{61}= \frac{p^t}{n}  \sum_{j \neq 0} R_t^2( \theta_j),
\end{align}
where $R_t= \frac{p-1}{p} U_t+ \frac{2}{p} T_t$. We write \eqref{61} as
\begin{equation}\label{63}
W_2(t)= p^t \mu_X(R_t^2),
\end{equation}
where $\mu_X$ is the density of the eigenvalues on $[0, \pi]:$
\[\mu_X= \frac{1}{n} \sum_{j \neq  0 } \delta_{\theta_j}.\]
For any sequence of Ramanujan graphs $X$, $\mu_X$ is known to converge to the Plancherel measure $\nu_p,$ as $n \rightarrow \infty$ \cite{AGV}. That is for a fixed  polynomial $R$ 
\[ \mu_X(R) \rightarrow \int_{0}^{\pi} R(\theta) d \nu_{p}(\theta)\]
as $n \rightarrow \infty$. Here the Plancherel, or Kesten measure, $\nu_p$ is;
\begin{equation}\label{new}
d \nu_p= \frac{2(p+1) \sin^2 \theta}{\pi[(p^{1/2} + p^{-1/2})^2 - 4 \cos ^2\theta]} d \theta.
\end{equation}

For $X$'s whose girth is at least $\delta \log_p n$, the calculation in Section \ref{bd}, which was used to establish the bounded window for these, yields that for $t<g/5,$
\[W_2(t) \sim (p+1)p^{t-1},\]
as $n\rightarrow \infty$. 
Hence for these $X$'s and in this range of $t$'s
 \begin{equation} 
\label{meas}\mu_{X}(R^2_t) \sim \frac{p+1}{p} ,\end{equation}
as $n \rightarrow \infty$. 
One can check that the $R_t$'s are orthogonal polynomials for the measure $\nu_p$ on $[0,\pi] $ (see \cite{CA} for example) and that for $t \geq 1$
\begin{equation}\label{normal}
\int^{\pi}_0 R^2_t(\theta) d \nu_p(\theta)= \frac{p+1}{p}
\end{equation}

Thus \eqref{meas} reads that the large girth $X$'s and in the range $t< g/5$ 
\begin{equation} \label{69}\mu_X(R^2_t) \sim \nu_p(R^2_t)= \frac{p+1}{p},\end{equation}
as $n \rightarrow \infty$. Our conjecture is that \eqref{69}  holds in general for any sequence of Ramanujan graphs and in the larger range $t< 2 \log_p n$. From \eqref{63} the conjecture is equivalent to 
\[W_2(t) \sim N(t),\]
for $t< 2 \log_p n$ as $n \rightarrow \infty$.
In the forthcoming paper \cite{SZ}, Conjecture \ref{conj} is proven for various families of arithmetical Ramanujan graphs, such as the ones discussed in \cite{CFLMP}.

\section{Acknowledgements} We would like to thank Eyal Lubetzky for his comments and insights concerning cutoff for the NBRW.
\bibliographystyle{plain}
\bibliography{Ram_window}
\end{document}